\documentclass[11pt]{article}

\usepackage[latin1]{inputenc}
\usepackage{t1enc}
\usepackage{amsmath, amssymb, amsfonts, amsthm, amsopn,epsfig}
\usepackage[none]{hyphenat}
\usepackage{tikz}
\usetikzlibrary{math,calc,intersections,patterns}

\usetikzlibrary{positioning,arrows,shapes,decorations.markings,decorations.pathreplacing,matrix}
\tikzstyle{vertex}=[circle,draw=black,fill=black,inner sep=0,minimum size=5pt,text=white,font=\footnotesize]
\usepackage{float}
\usepackage{graphicx}
\usepackage{caption}
\usepackage[toc,page]{appendix}
\usepackage{epstopdf}
\usepackage{etoolbox}
\usepackage{mathtools}
\usepackage[colorlinks=true]{hyperref}
\usepackage{dsfont}
\hypersetup{
	colorlinks=true,
	linkcolor=blue,
	filecolor=magenta,      
	urlcolor=cyan,
	citecolor=blue
}
\usepackage{cleveref}
\usepackage[top=25mm, bottom=25mm, left=25mm, right=25mm]{geometry}
\usepackage{comment}

\theoremstyle{plain}
\newtheorem{theorem}{Theorem}

\newtheorem{claim}[theorem]{Claim}
\newtheorem{lemma}[theorem]{Lemma}

\theoremstyle{definition}
\newtheorem{definition}{Definition}
\newtheorem{rmk}[theorem]{Remark}

\DeclareMathOperator{\supp}{supp}

\title{Explicit constructions of optimal blocking sets and minimal codes}
\author{Anurag Bishnoi\thanks{Delft University of Technology, \emph{e-mail}: \textbf{A.Bishnoi@tudelft.nl}} 
~and Istv\'an Tomon\thanks{Ume\r{a} University, \emph{e-mail}: \textbf{istvantomon@gmail.com}, Research supported in part by the Swedish Research Council grant VR 2023-03375.}}
\date{}

\begin{document}
	\sloppy 
	
	\maketitle
\begin{abstract}
 A strong $s$-blocking set in a projective space is a set of points that intersects each codimension-$s$ subspace in a spanning set of the subspace. We present an explicit construction of such sets in a $(k - 1)$-dimensional projective space over $\mathbb{F}_q$ of size $O_s(q^s k)$, which is optimal up to the constant factor depending on $s$. This also yields an optimal explicit construction of affine blocking sets in $\mathbb{F}_q^k$ with respect to codimension-$(s+1)$ affine subspaces, and of $s$-minimal codes. Our approach is motivated by a recent construction of Alon, Bishnoi, Das, and Neri of strong $1$-blocking sets, which uses expander graphs with a carefully chosen set of vectors as their vertex set. The main novelty of our work lies in constructing specific hypergraphs on top of these expander graphs, where tree-like configurations correspond to strong $s$-blocking sets. We also discuss some connections to size-Ramsey numbers of hypergraphs, which might be of independent interest.
\end{abstract}

\section{Introduction}

A blocking set in a finite projective or affine space is a set of points that intersects every subspace of a given dimension non-trivially. 
In graph terminology, these are vertex covers of the hypergraphs whose vertex sets are points of the finite space and edges are the subspaces of the fixed dimension. 
These objects have been studied extensively in finite geometry \cite{blokhuis2011blocking, brouwer1978blocking}, and many connections have been found with related areas like coding theory. 
A natural strengthening of this object, that incorporates more geometrical structure, is a \textit{strong} $s$-\textit{blocking set}: a set $B$ of points in the projective space $\mathrm{PG}(k - 1, q)$, that meets every subspace $S$ of codimension-$s$ in a set of points that spans $S$. 
We can also think of $B$ as a collection of pairwise linearly independent vectors 
in $\mathbb{F}_q^k$ such that the span $\langle B \cap S \rangle$ is equal to $S$ for every vector subspace $S$ of (vector) dimension $k - s$. 
For $s = 1$, these objects are simply known as strong blocking sets \cite{davydov2011linear, heger2021short}, and they have also been studied under the names of cutting blocking sets \cite{alfarano2022geometric, bonini2021minimal} and generating sets \cite{fancsali2014lines, fancsali2016higgledy}. 
Finding the smallest size of a strong blocking set is a major open problem in finite geometry, whose importance has increased significantly in the last few years due to its connection to finding short minimal codes \cite{alfarano2022geometric, tang2021full}. 
More recently, it has been shown in \cite{bishnoi2024blocking} that strong blocking sets of size $n$ over $\mathbb{F}_3^k$ are equivalent to linear trifferent codes of length $n$ and dimension $k$, thus tying them to another important open problem in information theory, \textit{the trifference problem}. 

In \cite[Lemma 1.2]{bishnoi2024blocking}, it was shown that a strong $s$-blocking set of size $n$ in the projective space $\mathrm{PG}(k -1, q)$ gives rise to an affine blocking set of size $(q - 1)n + 1$ in $\mathbb{F}_q^k$ with respect to the $(k - s - 1)$-dimensional affine subspaces. 
An application of the polynomial method \cite{brouwer1978blocking}, combined with a geometric argument \cite[Section 3]{ball2011polynomial}, then implies that a strong $s$-blocking set in $\mathrm{PG}(k-1, q)$ must be of size 
at least $(q^{s+1} - 1)(k - s)/(q - 1)$. 
For $s = 1$, this lower bound was improved in \cite[Theorem 1.4]{bishnoi2024blocking} to $c_q (q + 1) ( k - 1)$ for a constant $c_q > 1$ that does not have a closed formula, but it can be computed by solving an equation involving the $q$-ary entropy function and the linear-programming bound from coding theory. 
The best upper bounds are obtained by taking a random collection of $s$-dimensional (projective) subspaces \cite[Remark 3.4]{bishnoi2024blocking} and they are of the order $(q^{s + 1} - 1)((s + 1)(k - s - 1))/(q - 1)$, which is roughly $s + 1$ times the lower bound. 
Therefore, it is still an open problem to determine the smallest size of a strong $s$-blocking set for every $s \geq 1$.

As the random construction suggests, a useful way of constructing strong $s$-blocking sets is by using a collection of $s$-dimensional projective subspaces in $\mathrm{PG}(k - 1, \mathbb{F})$, for an arbitrary field $\mathbb{F}$, whose union meets every codimension-$s$ subspace in a spanning set; this is also known as a `higgledy-piggledy' arrangement of subspaces \cite{fancsali2016higgledy} and it has connections with subspace designs \cite{guruswami2016explicit}. 
Over characteristic $0$ fields, or finite fields of sufficiently large size with respect to the dimension, we know an explicit construction of such higgledy-piggledy subspaces \cite{fancsali2016higgledy}, and thus strong $s$-blocking sets, which is a constant factor away (for a fixed $s$ and $k$) from the lower bound. 
The main challenge is then to give constructions in $\mathrm{PG}(k - 1, q)$ where both $q$ and $s$ are fixed, and $k \rightarrow \infty$. This is also the setting that is more interesting from a coding-theoretic point of view where we need a fixed alphabet size. 

In \cite{alon2024strong}, an explicit construction of size $O(q k)$ was obtained for strong $1$-blocking sets in $\mathrm{PG}(k - 1, q)$ using constant-degree expander graphs. 
This solved the main open problem on constructing short minimal codes because of their known equivalence with strong blocking sets \cite{alfarano2022geometric, tang2021full}. 
Here, a linear code $C \leq \mathbb{F}_q^n$ is called \emph{minimal} if there are no two linearly independent codewords $x, y$ in $C$ for which the support of $x$ is a  subset of the support of $y$. 
A natural generalization of this, called \emph{$s$-minimal code}, is a linear code $C$ that contains no two distinct $s$-dimensional subspaces $U, V$ for which the support of $U$ is contained in the support of $V$. Here, the support of a vector subspace is the set of coordinate positions where at least one vector in the subspace has a non-zero entry. 

\subsection{Our results}

We give the first explicit constructions of strong $s$-blocking sets, whose size is optimal as a function of both $q$ and $k$ for $s > 1$. In particular, we explicitly construct  strong $s$-blocking sets of size $O_s(q^s k)$ in $PG(k-1,q)$. This is indeed optimal up to the dependence of the constant factor on $s$ by the aforementioned lower bound $(q^{s+1}-1)(k-s)/(q-1)$. 

Our construction is a far reaching generalization of the approach of \cite{alon2024strong}, which constructs a strong $1$-blocking set. In \cite{alon2024strong}, the authors consider an expander graph $G$, whose vertices are $O(k)$ vectors in `general position'. Then the desired strong $1$-blocking set is  the collection of $1$-dimensional subspaces, that is, points of $\mathrm{PG}(k - 1, q)$, that are contained in the union of the $2$-dimensional subspaces spanned by the edges of $G$. We extend this approach to hypergraphs. We define an $(s+1)$-uniform hypergraph $H$, whose vertex set is a set of $O_s(k)$ vectors in $\mathbb{F}_q^k$ in `general position'. Then our strong $s$-blocking set $B$ is the collection of $1$-dimensional subspaces that are contained in the union of the $(s+1)$-dimensional subspaces spanned by the edges of $H$.  We build $H$ on top of an expander graph $G$ of degree $O_s(1)$, by including all $(s+1)$-tuples of vertices of distance at most $s+1$ from a given vertex in $G$. This ensures that $H$ has $O_s(k)$ edges, and thus $B$ has $O_s(q^sk)$ elements. We discuss expander graphs and the notion of `general position' in Section \ref{sect:background}. Given a codimension-$s$ subspace $L$, the intersection of $B$ and $L$ naturally gives rise to a certain subhypergraph in $H$. We argue that the existence of large tree-like configurations in this hypergraph ensures that $L\cap B$ spans $L$. These tree-like configurations and their relationship to spanning sets is discussed in Section \ref{sect:trees}.

First, we present a construction for $s=2$ to illustrate the core of our ideas, which can be found in Section \ref{sect:s=2}. Then, in Section \ref{sect:size-ramsey}, we explore connections of strong blocking sets to a popular topic of extremal combinatorics, known as size-Ramsey numbers. We use recent results of Letzter, Pokrovskiy and Yepremyan \cite{LPY} about the size-Ramsey numbers of tight paths to construct  strong $s$-blocking sets of size $O_s(q^sk)$ for large prime powers $q$ with respect to $s$. Here, unfortunately, the constant hidden by the $O_s(.)$ notation is astronomical as a function of $s$. We improve this in Section \ref{sect:large_prime}, where we present a self-contained construction of strong $s$-blocking sets of size at most $2^{O(s^2\log s)}q^sk$ for every prime power $q$ sufficiently large with respect to $s$. 
We treat small prime powers $q$ separately in Section \ref{sect:small_prime}, where we provide explicit constructions of size at most $q^{O(s^2)}k$ (which is of the promised form $O_s(q^sk)$ assuming $q$ is bounded by a function of $s$).

Finally, by a known duality between strong $s$-blocking sets and $s$-minimal codes (see Example 3.2 in \cite{xu2024}), these also provide the first explicit constructions of asymptotically optimal $s$-minimal codes. For completeness, we also discuss this relationship between strong $s$-blocking sets and $s$-minimal codes in Section~\ref{sec:coding_background}.

\section{Background}\label{sect:background}

\subsection{Expander graphs}
In this section, we define and collect some useful properties of expander graphs. 

\begin{definition}[Expander graph]
    An \emph{$(n,d,\lambda)$-graph} is a $d$-regular connected graph on $n$ vertices such that every eigenvalue of the adjacency matrix other than $d$ is at most $\lambda$ in absolute value.
\end{definition}

Expander graphs are of high importance in computer science \cite{AC,HLW}, coding theory \cite{alon2024strong,sipser1996expander}, and combinatorics \cite{LPY}, as witnessed by countless applications. A family of constant degree expanders usually refers to some infinite family of $(n_i,d,\lambda)$-expanders, where $n_i\rightarrow\infty$, and $\lambda<\varepsilon d$ for some $\varepsilon\in (0,1)$. The Alon-Boppana theorem tells us that an $(n,d,\lambda)$-graph can only exist if $\lambda\geq 2\sqrt{d-1}-o_n(1)$. Graphs achieving this theoretical barrier are called \emph{Ramanujan graphs} and the celebrated result of Lubotzky, Phillips and Sarnak \cite{LPS} gives explicit constructions of such graphs for an infinite family of parameters.

\begin{lemma}\label{lemma:Ramanujan}
Let $q>p$ be primes congruent to 1 modulo 4. Then there is an explicit construction of an $(n,p+1,2\sqrt{p})$-graph for some $n=\Theta(q^3)$.
\end{lemma}

Unfortunately, there are no known explicit constructions of Ramanujan graphs for every pair of parameters $(n,d)$. However, the previous lemma implies that for every $d$ and $n$ sufficiently large with respect to $d$, there is a $d_0$-regular Ramanujan graph on $n_0$ vertices, where $d_0=\Theta(d)$ and $n_0=(1+o(1))n$, by well known results on the distribution of primes. Alon \cite{alon2021explicit} presents further explicit and strongly explicit constructions of almost Ramanujan graphs for every $n$. For our purposes, given the parameter $s$, we need an infinite family of explicit $(n,d,\lambda)$-graphs, where $d=O_s(1)$ and $\lambda\leq \varepsilon d$ for some fixed $\varepsilon=\varepsilon(s)>0$. The existence of such a family is guaranteed by Lemma \ref{lemma:Ramanujan}. 

One of the key technical results about expander graphs is the \emph{Expander mixing lemma}. Given a subset $U$ of the vertices of a graph $G$, $G[U]$ denotes the subgraph of $G$ induced on $U$. Also, if $U$ and $V$ are disjoint sets of vertices, then $G[U,V]$ is the induced bipartite subgraph of $G$ with parts $U$ and $V$. 

\begin{lemma}[Expander mixing lemma \cite{AC}]
Let $G$ be an $(n,d,\lambda)$-graph. Then for every $U\subset V(G)$,
$$\left|2e(G[U])-\frac{d}{n}|U|^2\right|\leq \lambda |U|.$$
Also, for every pair of disjoint $U,V\subset V(G)$,
$$\left|e(G[U,V])-\frac{d}{n}|U||V|\right|\leq \lambda \sqrt{|U||V|}.$$
\end{lemma}

\noindent
One of the simple consequences of the Expander mixing lemma is the following statement.

\begin{lemma}\label{lemma:component}
Let $G$ be an $(n,d,\lambda)$-graph, and let $H$ be a subgraph of $G$ with at least $\lambda n$ edges. Then $H$ has a connected component of size at least $e(H)/d$.
\end{lemma}

\begin{proof}
Let $U_1,\dots,U_s$ be the vertex sets of connected components of $H$, and let $b=\max_{i\in [s]} |U_i|$. By the first part of the Expander mixing lemma, we have 
$$e(G[U_i])\leq \frac{d}{2n}|U_i|^2+\frac{\lambda}{2}|U_i|\leq \left(\frac{db}{2n}+\frac{\lambda}{2}\right)|U_i|$$ for $i=1,\dots,s$. Therefore,
$$e(H)\leq \sum_{i=1}^s e(G[U_i])\leq \sum_{i=1}^s\left(\frac{db}{2n}+\frac{\lambda}{2}\right)|U_i|= \frac{db+\lambda n}{2}.$$
Comparing the left and right hand side, we get $$b\geq \frac{2e(H)-\lambda n}{d}\geq \frac{e(H)}{d}.$$
\end{proof}

\noindent
Another useful corollary is the following.

\begin{lemma}\label{lemma:large_component}
    Let $G$ be an $(n,d,\lambda)$-graph, and let $U\subset V(G)$. Then $G[U]$ has a connected component of size at least $|U|-\frac{2\lambda n}{d}$.
\end{lemma}

\begin{proof}
Let $U_1,\dots,U_{s}$ be the vertex sets of connected components of $G[U]$, and let $b=\max_{i\in [s]} |U_i|$. Then by the Expander mixing lemma, $$e(G[U_i])\leq \frac{d}{2n}|U_i|^2+\frac{\lambda}{2}|U_i|\leq \left(\frac{db}{2n}+\frac{\lambda}{2}\right)|U_i|.$$
On the other hand, $e(G[U])\geq \frac{d}{2n}|U|^2-\frac{\lambda}{2} |U|$. Hence,
$$\frac{d}{2n}|U|^2-\frac{\lambda}{2} |U|\leq \sum_{i=1}^{s}\left(\frac{db}{2n}+\frac{\lambda}{2}\right)|U_i|=\left(\frac{db}{2n}+\frac{\lambda}{2}\right)|U|.$$
From this, we get
$$b\geq |U|-\frac{2\lambda n}{d}.$$
\end{proof}

\noindent
Finally, we need the following technical result.

\begin{lemma}\label{lemma:star}
Let $G$ be an $(n,d,\lambda)$-graph, and let $U_0,\dots,U_{t}\subset V(G)$ be disjoint sets such that $|U_0|> t\lambda n /d$, and $|U_i|\geq \lambda n/d$ for $i=1,\dots,t$. Then there is a vertex $x\in U_0$ that sends an edge to every $U_i$, $i=1,\dots,t$.
\end{lemma}

\begin{proof}
The second part of the Expander mixing lemma implies that there is an edge between any two sets of size more than $b=\lambda n/d$. Hence, if $W_i\subset U_0$ is the set of vertices that send an edge to $U_i$, then $|W_i|\geq |U_0|-b$. But then $|\bigcap_{i=1}^t W_i|\geq |U_0|-bt>0$. Any vertex $x\in \bigcap_{i=1}^t W_i$ suffices.
\end{proof}

\subsection{Coding theory}\label{sec:coding_background}
In this section we define the basic notions from coding theory that we will need later, and prove the relation between $s$-minimal codes and strong $s$-blocking sets. 

An $[n, k, d]_q$ code is a $k$-dimensional vector subspace of $\mathbb{F}_q^n$ such that every non-zero vector in this subspace has Hamming weight at least $d$, that is, it has at least $d$ non-zero coordinates, and there is a vector with Hamming weight equal to $d$. 
The parameter $n$ is the length of the code, $k$ its dimension, and $d$ is the minimum distance of the code. 
There are two matrices associated with a $k$-dimensional code $C<\mathbb{F}_q^n$. A \emph{generator matrix} of $C$ is a $k\times n$ matrix, whose rows are the vectors of a basis of $C$. A \emph{check matrix} of $C$ is an $(n-k)\times n$ matrix, whose null-space is $C$.

The \emph{dual code} of $C$, denoted by $C^{\perp}$, is the code generated by the rows of the check matrix, or equivalently, the subspace orthogonal to $C$. It is easy to observe that a check matrix of $C$ is a generator matrix of $C^{\perp}$, and vice versa.

We now define minimal codes and for self-containment prove their equivalence with strong blocking sets, which was observed by Alfarano, Borello,  Neri, and  Ravagnani \cite{ABNR} for $s = 1$ but the proof of the general case is very similar. 
The support of a vector subspace $X$ of $\mathbb{F}_q^n$, denoted by $\mathrm{supp}(X)$, is the set of coordinate positions $i$ such that there is a vector $x \in X$ with $x_i \neq 0$.

\begin{definition}[Minimal codes]
An $[n, k, d]_q$ code $C$ is called $s$-minimal if for any two distinct $s$-dimensional subspaces $X, Y$ of $C$, $\mathrm{supp}(X) \not\subset \mathrm{supp}(Y)$ and $\mathrm{supp}(Y) \not\subset \mathrm{supp}(X)$.
\end{definition}

In other words, the supports of $s$-dimensional subspaces in an $s$-minimal code must form an antichain in the boolean poset $2^{[n]}$. 

\begin{theorem}
Let $G$ be a $k \times n$ matrix over $\mathbb{F}_q$ such that all of its columns are non-zero and pairwise linearly independent. 
Then the row-space of $G$ is an $s$-minimal code if and only if the columns of $G$ form a strong $s$-blocking set. 
\end{theorem}
\begin{proof}
   Let $C<\mathbb{F}_q^n$ be the row-space of $G$. An $s$-dimensional subspace $X$ of $C$ can be uniquely identified with the row space of $M G$, where $M$ is a full rank $s \times k$ matrix. 
    Moreover, the null-space of $M$ gives us a codimension $s$-subspace in $\mathbb{F}_q^k$ which contains the $i$-th column of $G$ if and only if the $i$-th coordinate is equal to $0$ in all codewords of $X$. 
    Let $\phi$ be the bijection from the codimension-$s$ subspaces of $\mathbb{F}_q^k$ to the $s$-dimensional subspaces of $C$ that maps the null space of $M$ to the row space of $MG$. 
    
    The columns do not form a strong $s$-blocking set in $\mathbb{F}_q^k$ if and only if there is a codimension-$s$ subspace $H$ such that the set of columns of $G$ that lie in $H$ are all contained in a subspace $T$ of one dimension less.
    Thus we can find a distinct codimension-$s$ subspace $H'<\mathbb{F}_q^k$ such that $H \cap H' = T$ and all columns of $G$ that are contained in $H$ are also contained in $H'$. 
    This happens if and only if $[n] \setminus \mathrm{supp}(\phi(H)) \subset [n] \setminus \mathrm{supp}(\phi(H'))$, which is equivalent to $C$ not being $s$-minimal. 
\end{proof}

\subsection{Points in general position}

Our constructions require a supply of $n=O_s(k)$ vectors  $W\subset \mathbb{F}_q^k$ with the following two properties: for some fixed $\varepsilon=\varepsilon(s)>0$, every set of $\varepsilon n$ elements of $W$ span the whole space, and every $s+1$ elements of $W$ are linearly independent. Explicit constructions of such sets can be found in coding theory, as follows. 

We need the following well known properties of generator and check matrices of a linear code. If $C$ is an $[n,k,d]_q$-code, and $H$ is a check-matrix of $C$, then any $d-1$ columns of $H$ are linearly independent. Moreover, if $G$ is a generator matrix of $C$, then any $n-d+1$ columns of $G$ span the whole space.
Therefore, the required set can be constructed by using the columns of the generator matrix of an $[n, k, d]_q$ code $C$ where $n = O_{s}(k)$, $d \geq (1 - \varepsilon) n + 1$ and the minimum distance of the dual code $C^\perp$ satisfies $d^\perp \geq s + 2$.

\begin{definition}[Rate and relative distance]
Let $\{n_i \}_{i \geq 1}$ be an increasing sequence of positive integers and suppose that there exist sequences $\{k_i\}_{i \geq 1}$ and $\{d_i\}_{i \geq 1}$ such that
for all $i \geq 1$ there is an $[n_i, k_i, d_i]_q$ code $\mathcal{C}_i$. 
Then $\{\mathcal{C}_i\}_{i \geq 1}$ is called an $(R,\delta)_q$-family of codes, where 
the rate $R$ of this family is defined as 
\[R \coloneqq \liminf_{i \rightarrow \infty} \frac{k_i}{n_i},\]
and the relative distance $\delta$ is defined as 
\[\delta \coloneqq \liminf_{i \rightarrow \infty} \frac{d_i}{n_i}.\]
\end{definition}

The main problem in coding theory is to study the trade-off between the rate and the relative distance of codes. 
A family of codes for which $R > 0$ and $\delta > 0$, is known as an \textit{asymptotically good code}. 
The Gilbert-Varshamov bound, which follows from a probablistic argument, shows the existence of such codes for every $\delta \in [0, 1 - 1/q)$ and $R = 1 - H_q(\delta)$, where
\[H_q(x) \coloneqq x \log_q(q - 1) - x \log_q(x)  - (1 - x) \log_q(1 - x),\]
is the $q$-ary entropy function.
The first \textit{explicit construction} of asymptotically good codes was given by  Justesen \cite{justesen1972class}, who showed that for every $0 < R < 1/2$, there is an explicit family of codes with rate $R$ and relative distance $\delta \geq (1 - 2R) H_q^{-1}\left(\frac{1}{2}\right)$. 
Improving the values of the rate $R$ and the relative distance $\delta$ has been an active area of research in coding theory since the 1970s.
One of the most significant developments was the use of modular curves to show that, for every square $q \geq 49$, there are explicit constructions of linear codes over $\mathbb{F}_q$ that beat the Gilbert-Varshamov  bound (see~\cite{couvreur2021algebraic, tsfasman2013algebraic} for some recent surveys on these constructions).
Another useful tool for constructing explicit codes is expander graphs \cite{sipser1996expander}.

We first note that these well-known explicit constructions of asymptotically good codes \cite{alon1992construction, justesen1972class, tsfasman2013algebraic, tsfasman1982modular} and the discussion above implies the following construction where we do not have any restriction on the dual distance. 
This will suffice for small prime powers. 

\begin{lemma}\label{lemma:constr2}
There exists $\delta>0$ and $C>0$ such that the following holds for every prime power $q$, and sufficiently large $k$. There exists an explicit construction of a set $W\subset \mathbb{F}_q^k$ of size at most $Ck$ such that any $(1-\delta)|W|$ elements of $W$ span $\mathbb{F}_q^k$.
\end{lemma}

For larger prime powers, we need the following stronger result. 

\begin{lemma}\label{lemma:constr1}
For every integer $s$ and $\varepsilon_0>0$, there exists $C$ such that the following holds. 
Let $q$ be a prime power and $k$ be an integer, both sufficiently large with respect to $s$ and $\varepsilon_0$. Then there exists an explicit construction of a set $W\subset \mathbb{F}_q^k$ of size at most $Ck$ such that any $s+1$ elements of $W$ are linearly independent, and any $\varepsilon_0 |W|$ elements of $W$ span $\mathbb{F}_q^k$. In case $q$ is a square or $k<\varepsilon_0 q$, $W$ is strongly explicit, otherwise $W$ can be constructed in $k^{O_{s,\varepsilon}(1)}$ time.
\end{lemma}
\begin{proof}
    In case $k<\varepsilon_0 q$, Reed-Solomon codes suffice. In particular, take $S\subset \mathbb{F}_q$ of size $k/\varepsilon_0$, and let the elements of $W$ be $(1,\alpha,\dots,\alpha^{k-1})$ for $\alpha\in S$. It is easy to check that all conditions are satisfied with $C=1/\varepsilon_0$.
    
    When $q$ is a square, for every $R\in (0,1)$, asymptotic Goppa codes provide an $(R,\delta)_q$-family of codes for some $\delta$ satisfying $R+\delta \geq  1-1/A(q)$, where $A(q)=\sqrt{q}-1$ is Ihara's constant (see Section 2.9 in \cite{hoholdt1998algebraic}). 
    From Theorem 2.69 and 2.71 in \cite{hoholdt1998algebraic} it can be derived that the relative distance of the dual code satisfies 
$\delta^\perp = \delta + 2R - 1 \geq R - 1/A(q)$.
Therefore the columns of the generator matrix of an $[n, k, d]_q$ code from this $(R,\delta)_q$-family give rise to non-zero vectors 
$v_1, \dots, v_n$ in $\mathbb{F}_q^k$, such that any $r$ of them are linearly independent and any $n - d + 1$ of them span the whole space, with 
\[\lim_{n \rightarrow \infty} \frac{k}{n} = R,\]
\[\lim_{n \rightarrow \infty} \frac{n - d}{n} \leq R + 1/A(q),\]
\[\lim_{n \rightarrow \infty} \frac{r}{n} \geq R - 1/A(q).\] 
For our construction we pick $q$ sufficiently large such that $1/A(q)<\varepsilon_0/4$, and set $R=2/A(q)$. Then $R+1/A(q)<3\varepsilon_0/4$ and $R-1/A(q)\geq 1/A(q)$, so if $k$ is sufficiently large, there is a set of vectors $W\subset \mathbb{F}_q^k$ with the desired properties. 

We provide the construction for general $q$ in the Appendix.
\end{proof}

\begin{rmk}
    An alternative construction can be obtained via asymptotically-good self-dual codes \cite[Theorem 1.6]{Stichtenoth06}, which also works for large enough $q$. 
    For $q = 2$ there is another construction of self-dual codes given in \cite{ALT}, motivated by quantum error correction, that we can use. 
\end{rmk}

\section{The hypergraph of linear combinations}\label{sect:trees}

Given a set $U\subset \mathbb{F}_q^k$, we denote by $\mbox{span}(U)$ or $\langle U\rangle$ the linear subspace spanned by $U$, and $\dim(U)$ denotes the dimension of $\mbox{span}(U)$. Say that a linear combination $a_1 v_1+\dots+a_rv_r$ of the set of vectors $\{v_1,\dots,v_r\}$ is \emph{proper} if none of the coefficients $a_i$ is equal to 0.

Given a set of vectors $W\subset \mathbb{F}_q^k$ and a target set $R\subset \mathbb{F}_q^k$, we associate the pair $(W,R)$ with a hypergraph as follows.

\begin{definition}[Proper linear combination hypergraph]
    For $W,R\subset \mathbb{F}_q^k$, let $H_{W\rightarrow R}$ be the hypergraph whose vertex set is $W$, and a subset $X\subset W$ forms an edge if there is a proper linear combination of the elements of $X$  contained in $R$. We refer to this hypergraph as the \emph{proper linear combination hypergraph} of $(W,R)$. 
\end{definition} 

\begin{definition}[Proper span]
    Given a hypergraph $H$ on vertex set $W\subset \mathbb{F}_q^k$, the \emph{proper span} of $H$, denoted by $\mbox{pspan}(H)$, is the set of all $v\in \mathbb{F}_q^k$ such that $v$ can be written as a proper linear combination of the elements of an edge of $H$.
\end{definition}

With these definitions in hand, our aim in this section is to prove that for certain tree-like hypergraphs $H$, if $H$ is a subhypergraph of $H_{W\rightarrow R}$, then $\mbox{pspan}(H)\cap R$ spans a subspace of large dimension in $R$. To this end, we first define our notion of tree-like hypergraphs.

\begin{definition}[$s$-tree-like hypergraph]
    Given a positive integer $s$, a hypergraph is \emph{$s$-bounded} if the size of every edge is  at most $s$. A hypergraph $H$ is \emph{$s$-tree-like} if it is $s$-bounded and there is an ordering $v_1,\dots,v_n$ of the vertices of $H$ such that the vertex $v_i$ in the subhypergraph of $H$ induced on $\{v_i,\dots,v_n\}$ has degree $1$ for $i=1,\dots,n-s+1$.
\end{definition}

 The main lemma of this section is the following.

\begin{lemma}\label{lemma:tree-like}
Let $W,R\subset \mathbb{F}_q^k$. If $H_{W\rightarrow R}$ contains an $s$-tree-like subhypergraph $H$ on vertex set $U$, then 
$$\dim(\mathrm{pspan}(H)\cap R)\geq \dim(U)-s+1.$$
\end{lemma}

\begin{proof}
 Let $v_1,\dots,v_n$ be an ordering of $U$ satisfying that the vertex $v_{\ell}$ in the subhypergraph of $H$ induced on $\{v_{\ell},\dots,v_n\}$ has degree $1$ for $\ell=1,\dots,n-s+1$. Let $e_{\ell}$ denote the unique edge containing $v_{\ell}$ contained in $\{v_{\ell},\dots,v_n\}$. Define the $(n-s+1)\times n$ sized matrix $M$ over $\mathbb{F}_q$ as follows. As $e_{\ell}$ is an edge of $H_{W\rightarrow R}$, there exist coefficients $a_1,\dots,a_n\in\mathbb{F}_q^n$ such that $\sum_{i=1}^n a_i v_i\in R$, and $a_i\neq 0$ if and only if $v_i\in e_{\ell}\subset \{v_{\ell},\dots,v_n\}$. Set $M(\ell,i)=a_i$. 

 Then $M$ is an upper triangular matrix such that $M(\ell,\ell)\neq 0$ for $\ell=1,\dots,n-s+1$, thus $\mbox{rank}(M)=n-s+1$. Let $N$ be the $n\times k$ matrix, whose rows are the vectors $v_1,\dots,v_n$. Then $\mbox{rank}(N)=\dim(U)$ and the rows of the matrix $M\cdot N$ are elements of $\mbox{pspan}(H)\cap R$. Thus, 
 $$\dim(\mbox{pspan}(H)\cap R)\geq \mbox{rank}(M\cdot N)\geq \mbox{rank}(M)+\mbox{rank}(N)-n= \dim(U)-s+1.$$
 Here, the second inequality is due to Sylvester's rank inequality.
\end{proof}

\section{Constructions of strong blocking sets}
\subsection{Codimension 2}\label{sect:s=2}

As a warm-up, in this section, we present an explicit construction of a strong 2-blocking set of size $O(q^2 k)$ for every sufficiently large prime power $q$ and sufficiently large $k$. Given a set $B\subset \mathbb{F}_q^k$, let $\tilde{B}\subset PG(k-1,q)$ denote $B/\sim$, where $x\sim y$ if $x=cy$ for some $c\neq 0$.

\medskip

\noindent
\textbf{Construction.} Let $\alpha:=1/8$ and $d=258$, then $d-1$ is a prime congruent to 1 modulo 4. Let $\lambda:=2\sqrt{d-1}\approx 32.06$, and note that $\alpha>\lambda/d$ is satisfied.
\begin{itemize}
    \item Assume that $q$ is a prime power sufficiently large to satisfy the requirements of Lemma \ref{lemma:constr1} with $s=2$ and $\varepsilon_0=\alpha/2$. That is, there is an explicit construction  $W_0\subset \mathbb{F}_q^k$ of $n_0=O(k)$ vectors such that any subset of $W_0$ of size at least $\alpha n_0/2$ spans $\mathbb{F}_q^k$, and any three vectors in $W_0$ are linearly independent.
    
    \item Let $n_0\geq n=(1+o(1))n_0$ for which there is an explicit construction $G$ of an $(n,d,\lambda)$-graph. This exists by Lemma \ref{lemma:Ramanujan}. Associate the vertices of $G$ with an arbitrary $n$ element subset $W$ of $W_0$. Note that 
    any $\alpha n$ elements of $W$ span $\mathbb{F}_q^k$, and any $3$ elements of $W$ are linearly independent.
    
    \item A \emph{cherry} in $G$ is a 3 element set $\{x,y,z\}$ such that $xy,xz\in E(G)$. Let $H$ be the 3-uniform hypergraph, whose edges are the cherries of $G$.

    \item  Let $B=\bigcup_{f\in E(H)}\mbox{span}(f)$.
\end{itemize}   

\begin{theorem}
$\tilde{B}$ is a strong 2-blocking set of size $O(q^2 k)$ in $PG(k-1,q)$.
\end{theorem}

\begin{proof}
As $G$ is $d$-regular, it contains at most $nd^2$ cherries, so $|B|\leq q^3 n d^2=O(q^3 n)$. Therefore, $|\tilde{B}|=O(q^2 k)$.

Next, we prove that for every $(k-2)$-dimensional subspace $L<\mathbb{F}_q^k$, the set $L\cap B$ spans $L$. Define the hypergraph $F$ on vertex set $W$ as follows: for every cherry $\{x_1,x_2,x_3\}$, there exists at least one non-zero point $a_1x_1+a_2x_2+a_3x_3$ in $L$ for some $a_1,a_2,a_3\in\mathbb{F}_q$. Here, we are using that any three elements of $W$ are linearly independent. Add to the hypergraph $F$ the set  $\{x_i:a_i\neq 0\}$ as an edge. Then $F$ is $3$-bounded, and each cherry contributes a 1, 2, or 3 element edge to $F$. Furthermore, $\mbox{pspan}(F)\subset B$ and  $F$ is a subhypergraph of the proper linear combination hypergraph $H_{W\rightarrow L}$.

 Let $A\subset W$ be the set of one-element edges of $F$. Then $|A|< \alpha n$, otherwise $\dim(A)=k$, which is impossible as $A\subset L$. Let $G'$ be the subgraph of $G$ composed of those edges, that are also edges of $F$. Then $G'$ has no component of size more than $\alpha n$. Indeed, otherwise, if $T$ is a tree of size $\alpha n$ in $G'$, then $T$ is a $2$-tree-like subhypergraph of $F$ satisfying $\dim(V(T))=k$. Hence, by Lemma \ref{lemma:tree-like}, $$\dim(\mbox{pspan}(F)\cap L)\geq \dim(V(T))-1>\dim(L),$$ contradiction. As $G'$ has no component of size at least $\alpha n$, Lemma \ref{lemma:component} implies that $e(G')\leq \alpha d n$.

 Finally, let $G^*$ be the  subgraph of $G$ composed of those edges that are disjoint from $A$, and are not contained in $G'$. Then 
 $$e(G^*)\geq e(G)-|A|d-e(G')\geq \frac{dn}{2}-\alpha d n-\alpha d n\geq \frac{dn}{4}.$$
 Using Lemma \ref{lemma:component} again, we conclude that $G^*$ contains a connected component of size at least $n/4$. Let $T$ be a spanning tree of $G^*$ and let $y_1,\dots,y_m$ be an enumeration of the vertices of $T$ such that $y_i$ is a leaf of the subtree $T[y_i,\dots,y_m]$. We define a 3-tree-like subhypergraph $K$ of $F$ as follows. For every $i\in \{1,\dots,n-2\}$, there exists $i<j<j'$ such that $y_iy_j$ and $y_jy_{j'}$ are edges of $T$. But then $\{y_i,y_j,y_{j'}\}$ is a cherry of $G$, and none of the sets $\{y_i\}, \{y_j\},\{y_{j'}\}, \{y_i,y_j\}, \{y_jy_{j'}\}$ is an edge of $F$, so  at least one of $\{y_i,y_{j'}\}$ or $\{y_i,y_j,y_{j'}\}$ is an edge of $F$. Add this edge to $K$. 
 Then, $K$ is $3$-tree-like on vertex set $\{y_1,\dots,y_m\}\subset W$, where $m\geq n/4>\alpha n$, so $$\dim(\mbox{pspan}(K)\cap L)\geq \dim(K)-2=k-2,$$ by Lemma \ref{lemma:tree-like}. In conclusion, $\mbox{pspan}(K)\cap L\subset B\cap L$, thus $B\cap L$ spans $L$. See Figure \ref{fig:1} for an illustration of $K$.
 \end{proof}

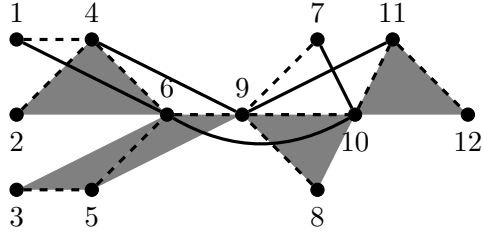
\begin{figure}
    \centering
    \begin{tikzpicture}
        \fill[gray] (4,0) --(2.5,0) -- (3,1);
        \fill[gray] (2,-1) --(1,0) -- (2.5,0) ;
        \fill[gray] (-1,-1) --(1,0) -- (0,0)  ;
        \fill[gray] (-1,1) --(-2,0) -- (0,0)  ;
        \fill[gray] (0,0) --(-2,-1) -- (-1,-1)  ;

        \node[vertex,label=above:{$1$}] (v1) at (-2,1) {};
        \node[vertex,label=below:{$2$}] (v2) at (-2,0) {};
        \node[vertex,label=below:{$3$}] (v3) at (-2,-1) {};
        \node[vertex,label=above:{$4$}] (v4) at (-1,1) {};
        \node[vertex,label=below:{$5$}] (v5) at (-1,-1) {};
        \node[vertex,label=above:{$6$}] (v6) at (0,0) {};
        \node[vertex,label=above:{$7$}] (v7) at (2,1) {};
        \node[vertex,label=below:{$8$}] (v8) at (2,-1) {};
        \node[vertex,label=above:{$9$}] (v9) at (1,0) {};     
        \node[vertex,label=below:{$10$}] (v10) at (2.5,0) {};     
        \node[vertex,label=above:{$11$}] (v11) at (3,1) {};
        \node[vertex,label=below:{$12$}] (v12) at (4,0) {};

        \draw[dashed,very thick] (v11) -- (v12) ;
        \draw[dashed,very thick] (v10) -- (v11) ;
        \draw[dashed,very thick] (v9) -- (v10) ;
        \draw[dashed,very thick] (v8) -- (v9) ;
        \draw[dashed,very thick] (v7) -- (v9) ;
        \draw[dashed,very thick] (v6) -- (v9) ;
        \draw[dashed,very thick] (v5) -- (v6) ;
        \draw[dashed,very thick] (v4) -- (v6) ;
        \draw[dashed,very thick] (v3) -- (v5) ;
        \draw[dashed,very thick] (v2) -- (v4) ;
        \draw[dashed,very thick] (v1) -- (v4) ;  

        \draw[very thick] (v9) -- (v11) ;
        \draw[very thick] (v7) -- (v10) ;
        \draw[very thick] (v6) edge[bend right] (v10) ;
        \draw[very thick] (v4) -- (v9) ;
        \draw[very thick] (v1) -- (v6) ;

    \end{tikzpicture}
    \caption{An illustration of a 3-tree-like subhypergraph $K$ we might find. The dashed edges are the edges of the tree $T$, while the black and gray edges (which have size 2 and 3, respectively) are the edges of $K$. The edges of $K$ in order are $\{1,6\}, \{2,4,6\}, \{3,5,6\}, \{4,9\}, \{5,6,9\},\{6,10\},$ $\{7,10\}, \{8,9,10\},\{9,11\}, \{10,11,12\}$.}
    \label{fig:1}
\end{figure}

\subsection{Explicit constructions via size-Ramsey numbers}\label{sect:size-ramsey}

In this section, we present an explicit construction of a strong $s$-blocking set  of size $O_s(q^{s}k)$ in $PG(k-1,q)$, assuming $q$ is sufficiently large with respect to $s$. Here, the constant hidden by the $O_s(.)$ notation is astronomical as a function of $s$, which will be greatly improved in Section \ref{sect:large_prime} by a different construction. However, the construction presented in this section highlights  some interesting connections to extremal combinatorics, in particular to the theory of \emph{size-Ramsey numbers}, so we believe it is beneficial to present it.

Given $r$-uniform hypergraphs $G,H_1,H_2$, we write $(H_1,H_2)\rightarrow G$ if for any coloring of the edges of $G$ with red or blue, $G$ contains either a red copy of $H_1$ or a blue copy of $H_2$. The \emph{size-Ramsey number} of $(H_1,H_2)$ is the smallest number of edges in an $r$-uniform hypergraph $G$ for which $(H_1,H_2)\rightarrow G$, and it is denoted by $\hat{r}(H_1,H_2)$. In case $H_1=H_2=H$, we may write simply $H\rightarrow G$ instead of $(H,H)\rightarrow G$, and $\hat{r}(H)$ instead of $\hat{r}(H,H)$. 

This notion was first introduced by Erd\H{o}s, Faudree, Rousseau and Schelp in 1978 \cite{EFRS} as a natural extension of the usual hypergraph Ramsey numbers. We highlight that the traditional Ramsey number $r(H_1,H_2)$ of $(H_1,H_2)$ is the minimal number of vertices in a graph $G$ for which $(H_1,H_2)\rightarrow G$ (in which case we can always assume that $G$ is a complete graph). One of the earliest results concerning size-Ramsey numbers is due to Beck \cite{beck}, who showed that if $P_n$ is the path of length $n$, then $\hat{r}(P_n)=O(n)$. In the past four decades, this area went through tremendous growth, and the following theorem of Letzter,  Pokrovskiy, and Yepremyan \cite{LPY} generalizes a long list of results in the topic.  The $r$-uniform \emph{tight path} of length $\ell$, denoted by $P_{r,\ell}$, is a sequence of $\ell+r-1$ vertices $v_1,\dots,v_{\ell+r-1}$ such that $\{v_{i},\dots,v_{i+r-1}\}$ is an edge for $i=1,\dots,\ell$. See Figure \ref{fig:tight_path} for an illustration. The \emph{$t$-power} of such a tight path, denoted by $P_{r,\ell}^t$, is the $r$-uniform hypergraph whose edges are all the $r$-sets contained in some interval of length $r+t-1$ among $v_1,\dots,v_{\ell+r-1}$. 

	\begin{figure}
	\begin{center}
		\begin{tikzpicture}[scale=3]
			\node[vertex,minimum size=4pt,label=below:{\small $1$}] at (0,0) {};
			\node[vertex,minimum size=4pt,label=above:{\small $2$}] at (0.25,0.5) {};
			\node[vertex,minimum size=4pt,label=below:{\small $3$}] at (0.5,0) {};
			\node[vertex,minimum size=4pt,label=above:{\small $4$}] at (0.75,0.5) {};
			\node[vertex,minimum size=4pt,label=below:{\small $5$}] at (1,0) {};
			\node[vertex,minimum size=4pt,label=above:{\small $6$}] at (1.25,0.5) {};
			\node[vertex,minimum size=4pt,label=below:{\small $7$}] at (1.5,0) {};
			\node[vertex,minimum size=4pt,label=above:{\small $8$}] at (1.75,0.5) {};
			\draw (0,0) -- (1.5,0);
			\draw (0.25,0.5) -- (1.75,0.5);
			\draw (0,0) -- (0.25,0.5) -- (0.5,0) -- (0.75,0.5) -- (1,0) -- (1.25,0.5) -- (1.5,0) -- (1.75,0.5);
		\end{tikzpicture}
	\end{center}
 \caption{An illustration of a 3-uniform tight path of length 6.}
 \label{fig:tight_path}
\end{figure}
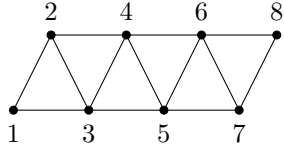

In \cite{LPY}, it is proved that for every $r,t$, there exists $C=C(r,t)>0$ such that 
$$\hat{r}(P_{r,n}^{t})\leq Cn.$$
For our applications, we need bounds on the size-Ramsey number of $(P_{r,n},K_{r,t})$, where $K_{r,t}$ denotes the $r$-uniform clique of size $t$. However, noting that $P_{r,n}^t$ contains both $P_{r,n}$ and $K_{r,t}$, we get the following immediate corollary of the previous theorem. For every $r,t$, there exists $C=C(r,t)>0$ such that 
$$\hat{r}(P_{r,n}, K_{r,t})\leq Cn.$$

Moreover, one can explicitly define a hypergraph $H$ with at most $Cn$ edges such that every two coloring of the edges of $H$ contains a monochromatic copy of $P_{r,n}^t$.  Given a graph $G$, let $G^u$ be the $u$-power of $G$, that is, any two vertices of distance at most $u$ are joined by an edge. Also, $G[K_D]$ denotes the $D$-blow-up of $G$, which is the graph in which every vertex is replaced by a clique of size $D$, and if there is an edge between two vertices in $G$, then all edges are drawn between the corresponding cliques of $G[K_D]$. We also use $K_r(G)$ to denote the $r$-uniform hypergraph, whose edges are the $r$-cliques of $G$. Let $G$ be an $\varepsilon$-expander, which is defined to be a graph that satisfies that between any pair of vertex sets of size at least $\varepsilon\cdot v(G)$, there is an edge. By the Expander Mixing lemma, $(n,d,\lambda)$-expanders satisfy this property with $\varepsilon=\lambda/d$. The following is a (very) special subcase of Theorem 6.1 in \cite{LPY}, where we only use that $P_{r,\ell}^t$ contains both a tight path of length $\ell$, and a clique of size $t$.

\begin{theorem}\label{thm:two_color}
    Let $r,t$ be integers, then if $\varepsilon^{-1}$ and $\Delta$ are sufficiently large with respect to $r,t$, and $\alpha,u,D$ are sufficiently large with respect to $\varepsilon,\Delta$, the following holds. Let $n$ be an integer, let $G$ be an $\varepsilon$-expander on at least $\alpha n$ vertices of maximum degree $\Delta$, and let $H=K_r(G^u[K_D])$. Then 
    $$(P_{r,n},K_{r,t})\rightarrow H.$$
\end{theorem}
We point out that if $G$ has maximum degree $d$, then $v(H)=Dv(G)$ and $e(H)\leq (2d)^{ur} v(H)$. Unfortunately, the dependence of the parameters $\varepsilon,\Delta,\alpha,u,D$ on $r$ and $t$ is quite poor. While in \cite{LPY} there are no explicit dependencies calculated, these parameters grow at least as fast as the usual Ramsey number of an $r$-uniform clique of size $t$. This number is at least $2^{2^{\dots 2^{\Omega(t)}}}$, where the tower is of height $r-2$. 

\medskip

To construct our strong $s$-blocking set in $\mathbb{F}_q^k$, we consider the $(s+1)$-uniform hypergraph $H$ given by the previous theorem (with parameters $t,n$ chosen appropriately with respect to $s$ and $k$) and associate its vertices with certain vectors in $\mathbb{F}_q^k$. Then our strong $s$-blocking set $B$  is the set of all linear combinations of the edges of $H$. Given a hyperplane $L$ of codimension $s$, $L$ defines a two coloring on the edges, where an edge is colored red if there is a proper linear combination of the vertices of the edge contained in $L$, otherwise it is colored blue. In our next lemma, we show that if any $s+1$ vertices of $H$ are linearly independent, then this coloring cannot contain a large blue clique.

\begin{lemma}\label{lemma:exactly_s+1}
For every positive integer $s$, there exists $R=R(s)$ such that the following holds for every  prime power $q$ larger than $s$. Let $U\subset \mathbb{F}_q^k$ be a set of $R$ vectors such that every $s+1$ elements of $U$ are linearly independent.  If $L$ is an $s$-codimensional subspace of $\mathbb{F}_q^k$, then $H_{U\rightarrow L}$ contains an edge of size exactly $s+1$.
\end{lemma}

\begin{proof}
We show that $R=(s+2)!$ suffices. Assume that $H_{U\rightarrow L}$ contains no edge of size $s+1$. For every $1\leq \ell\leq s$, let $F_{\ell}$ be the $\ell$-uniform hypergraph formed by the $\ell$-element edges of $H_{U\rightarrow L}$. 

\begin{claim}
$F_{\ell}$ contains no set of edges whose union has size $s+1$.
\end{claim}
\begin{proof}
Assume to the contrary that $e_1,\dots,e_m\in E(F_{\ell})$ such that $|e_1\cup\dots\cup e_m|=s+1$. Then for every $i=1,\dots,m$ and $x\in e_i$ there exists $a_{i,x}\in\mathbb{F}_q\setminus\{0\}$ such that 
$$y_i=\sum_{x\in e_i}a_{i,x} x\in L.$$
But then for arbitrary $b_1,\dots,b_m\in \mathbb{F}_q$, we also have
$$\sum_{i=1}^m b_i y_i=\sum_{x\in e_1\cup\dots \cup e_m} \left(\sum_{i: x\in e_i} b_i a_{i,x}\right)\cdot x\in L.$$
Using our assumption $q>s$, we can choose $b_1,\dots,b_m$ such that none of the coefficients $\sum_{i: x\in e_i} b_i a_{i,x}$ is $0$. Indeed, choosing $b_1,\dots,b_m$ randomly from the uniform distribution on $\mathbb{F}_q$, the probability that $\sum_{i:x\in e_i}b_i a_{i,x}=0$ is $1/q$, so with probability at least $1-s/q$, none of the coefficients are 0. This is a contradiction, as then $L$ contains a proper linear combination of the elements of $e_1\cup \dots\cup e_m$, so $e_1\cup\dots\cup e_m$ is an $s+1$ element edge of $H_{U\rightarrow R}$.
\end{proof}

In particular, the previous claim implies that $F_{\ell}$ contains no $(s+2-\ell)$-star, that is, $s+2-\ell$ edges, all of which intersect in a set of size $\ell-1$. As the number of $(\ell-1)$-sets is $\binom{R}{\ell-1}$, this gives
$$e(F_{\ell})\leq (s+1-\ell)\binom{R}{\ell-1}\leq s\binom{R}{\ell-1}.$$
Let $H_{\ell}$ be the $(s+1)$-uniform hypergraph of those edges $e$ which contain some edge of $F_{\ell}$. Then
$$e(H_{\ell})\leq \binom{R-\ell}{s+1-\ell}e(F_{\ell})\leq s\binom{R}{\ell-1}\binom{R-\ell}{s+1-\ell}\leq R^{s}.$$
Therefore, the total number of edges in the union of $H_{1},\dots,H_s$ is at most $s R^s$. Hence, if $R=(s+2)!$, then $sR^s<\binom{R}{s+1}$, which means that some $(s+1)$-element set $f\subset U$ is not an edge of any of the $H_1,\dots,H_s$. In other words, no linear combination of elements of $f$ is contained in $L$. But as the element of $f$ span a subspace of dimension $s+1$, this is impossible, contradiction.
\end{proof}

After these preparations, let us present our construction in detail.

\medskip

\noindent
    \textbf{Construction.} Let $t=R$ be given by Lemma \ref{lemma:exactly_s+1}, let $r=s+1$, and let $\varepsilon,\alpha>0$  and integers $u,D$ be given by Theorem \ref{thm:two_color}. Let $d=p+1$, where $p$ is the smallest prime congruent to 1 mod 4 larger than $16/\varepsilon^2$, and let $\lambda=2\sqrt{d-1}$. Note that $\lambda/d< \varepsilon$.

    \begin{itemize}
        \item  Let $q$ be a prime power and $k$ be an integer, both sufficiently large such that Lemma \ref{lemma:constr1} holds with $\varepsilon_0=1/(2D\alpha)$. That is, there is a set $W_0\in \mathbb{F}_q^k$ of $m_0=O_{s,\varepsilon_0}(k)$ vectors such that any $m_0/(2D\alpha)$ elements of $W_0$ span $\mathbb{F}_q^k$, and any $s+1$ are linearly independent. 

        \item Let $n_0=m_0/D$, and let $G$ be an explicit construction of an $(n, d,\lambda)$-graph for some $n_0>n=(1+o(1))n_0$. This exists by Lemma \ref{lemma:Ramanujan}. Note that $G$ is an $\varepsilon$-expander.

        \item Let $H=K_r(G^u[K_D])$, then $m:=v(H)=Dn>m_0/2$. Associate an $m$ element subset $W$ of $W_0$ arbitrarily with the vertices of $H$. Write $\ell=m/(\alpha D)$, and note that any $\ell$ elements of $W$ span $\mathbb{F}_q^k$, and any $s+1$ are linearly independent. Moreover, $(P_{r,\ell},K_{r,t})\rightarrow H$.
        
        \item Let $B=\bigcup_{f\in E(H)}\mbox{span}(f)$.
    \end{itemize}

\begin{theorem}
    $\tilde{B}$ is a strong $s$-blocking set of size $O_s(q^sk)$ in $PG(k-1,q)$.
\end{theorem}

\begin{proof}
The parameters $t,r,\varepsilon,\alpha,u,D,d$ only depend on $s$, which implies that $m=O_s(k)$. Hence, $|B|\leq |E(H)|q^{s+1}\leq (2d)^{ur} m=O_s(q^{s+1}k)$ and $\tilde{B}=O_s(q^s k)$.

We show that for every $s$-codimensional subspace $L<\mathbb{F}_q^k$, $B\cap L$ spans $L$. Color each edge $f\in E(H)$ with red or blue as follows. If there is a proper linear combination of the elements of $f$ that is contained in $L$, then color $f$ red, otherwise color $f$ blue. Note that the red edges are also edges of $H_{W\rightarrow L}$. By Lemma \ref{lemma:exactly_s+1}, this coloring contains no blue clique of size $t$. Hence, $H$ must contain a red tight path $P$ of size $\ell$. Then $\dim(V(P))=k$, as any $\ell$ elements of $V(H)$ span $\mathbb{F}_q^k$. But a tight path is also an $r$-tree-like hypergraph, hence Lemma \ref{lemma:tree-like} implies that $$\dim(\mbox{pspan}(H)\cap L)\geq \dim(V(P))-r+1=k-s.$$
Hence, as $\mbox{pspan}(H)\subset B$, we conclude that $B\cap L$ spans $L$.
\end{proof}

\subsection{Optimal constructions for large primes}\label{sect:large_prime}

 In this section, we present our best construction of a strong $s$-blocking set, assuming $q$ is sufficiently large with respect to $s$.

 \medskip

\noindent
\textbf{Construction.} Let $r=s+1$, and let $d=p+1$, where $p$ is the smallest prime congruent to 1 mod 4 such that $p>64r^2$, and let $\lambda=2\sqrt{d-1}$.
\begin{itemize}
    \item Let $q$ be a prime power and $k$ be an integer, both sufficiently large such that Lemma \ref{lemma:constr1} holds with $\varepsilon_0=1/(8rd^r)$. That is, there is a set $W_0\subset \mathbb{F}_q^k$  of $n_0=O_{s,\varepsilon}(k)$ vectors such that any $r$ elements of $W_0$ are linearly independent, and any $n_0/(8rd^r)$ elements of $W_0$ span $\mathbb{F}_q^k$.
    \item Let $n_0>n=(1+o(1))n_0$ be such that there exists an $(n,d,\lambda)$-graph $G$. This exists by Lemma \ref{lemma:Ramanujan}. Let $W$ be an $n$ element subset of $W_0$, and assign the elements of $W$ to the vertices of $G$. Note that any $n/(4rd^r)$ elements of $W$ span $\mathbb{F}_q^k$.
    \item Let $H=K_r(G^r)$, that is, the edges of $H$ are those $r$-tuples, whose every vertex is within distance at most $r$ from some vertex $x\in V(G)$. Note that $e(H)\leq n d^{r^2}$.
   
    \item Let $B=\bigcup_{f\in E(H)}\mbox{span}(f)$.
\end{itemize}

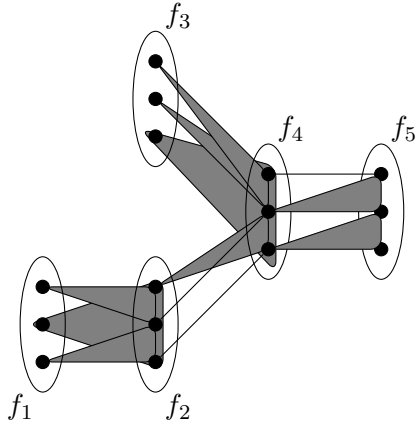
\begin{figure}
    \centering
    \begin{tikzpicture}
    \draw[rounded corners, fill=gray, fill opacity=0.5] (-1.7,1.1) -- (0.1,-0.8) -- (0.1,0.6) -- cycle;
    \draw[rounded corners, fill=gray, fill opacity=0.5] (-1.5,1.5) -- (0,0) -- (0,0.5) -- cycle;
    \draw[rounded corners, fill=gray, fill opacity=0.5] (-1.5,2) -- (0,0) -- (0,0.5) -- cycle;

    \draw[rounded corners, fill=gray, fill opacity=0.5] (-1.5,-1) -- (0,0) -- (0,-0.5) -- cycle;

    \draw[rounded corners, fill=gray, fill opacity=0.5] (-3.2,-1.5) -- (-1.4,-0.9) -- (-1.4,-2.1) -- cycle;
    \draw[rounded corners, fill=gray, fill opacity=0.5] (-3,-1) -- (-1.5,-1) -- (-1.5,-1.5) -- cycle;
    \draw[rounded corners, fill=gray, fill opacity=0.5] (-3,-2) -- (-1.5,-2) -- (-1.5,-1.5) -- cycle;
    
    \node[vertex] (f51) at (1.5,-0.5) {};
    \node[vertex] (f52) at (1.5,0) {};
    \node[vertex] (f53) at (1.5,0.5) {};
    \draw (f52) ellipse (0.3 and 0.9);
    \node at (1.8,1.1) {$f_5$};

    \node[vertex] (f41) at (0,-0.5) {};
    \node[vertex] (f42) at (0,0) {};
    \node[vertex] (f43) at (0,0.5) {};
    \draw (f42) ellipse (0.3 and 0.9);
    \node at (0.3,1.1) {$f_4$};

    \draw[rounded corners, fill=gray, fill opacity=0.5] (0,-0.5) -- (1.5,-0.5) -- (1.5,0) -- cycle; 
    \draw[rounded corners, fill=gray, fill opacity=0.5] (0,0) -- (1.5,0.5) -- (1.5,0) -- cycle;
    \draw (f43) -- (f53);

    \node[vertex] (f31) at (-1.5,1) {};
    \node[vertex] (f32) at (-1.5,1.5) {};
    \node[vertex] (f33) at (-1.5,2) {};
    \draw (f32) ellipse (0.3 and 0.9);
    \node at (-1.2,2.6) {$f_3$};

    \node[vertex] (f21) at (-1.5,-1) {};
    \node[vertex] (f22) at (-1.5,-1.5) {};
    \node[vertex] (f23) at (-1.5,-2) {};
    \draw (f22) ellipse (0.3 and 0.9);
    \node at (-1.2,-2.6) {$f_2$};

    \draw (f22) -- (f42) ;
    \draw (f23) -- (f41) ;

    \node[vertex] (f11) at(-3,-1) {};
    \node[vertex] (f12) at (-3,-1.5) {};
    \node[vertex] (f13) at (-3,-2) {};
    \draw (f12) ellipse (0.3 and 0.9);
    \node at (-3.3,-2.6) {$f_1$};

    \end{tikzpicture}
    \caption{An illustration for the proof of Theorem \ref{thm:large_prime}. An example of an $\ell=4$-tree-like hypergraph we might build on the union of the edges $f_1,\dots,f_5$.}
    \label{fig:3}
\end{figure}

\begin{theorem}\label{thm:large_prime}
$\tilde{B}$ is a strong $s$-blocking of size $2^{O(s^2\log s)}kq^s$ in $PG(k-1,q)$.
\end{theorem}

\begin{proof}
The size of $B$ is at most $e(H)q^{s+1}\leq nd^{r^2} q^{s+1}=2^{O(s^2\log s)} k q^{s+1}.$ Hence, $|\tilde{B}|=2^{O(s^2\log s)} k q^{s}$.

We show that if $L\subset \mathbb{F}_q^k$ has codimension $s$, then $L\cap B$ spans $L$.  For an integer $t$ and vertex $x\in V(G)$, let $B_t(x)$ denote the set of vertices of $G$ of distance at most $t$ from $x$. Given an integer $\ell \in [r]$ and a vertex  $x\in V(G)$, say that 
\begin{itemize}
    \item[(i)] $x$ is \emph{$\ell$-full} if  for every $\ell$-tuple of vertices $f\subset B_{\ell-1}(x)$, the span of $f$ intersects $L$ non-trivially.
\end{itemize}
Observe that every vertex $x$ is $(s+1)$-full, as $L$ has codimension $s$. Also, a vertex is $1$-full if and only if $x\in L$. If a vertex $x$ is $1$-full, color it with color $1$, otherwise, color $x$ with an integer $\ell\in \{2,\dots,s+1\}$ such that $x$ is $\ell$-full, but not $(\ell-1)$-full. Then every vertex is colored with some color. Therefore, there exists a color $\ell\in [r]$ such that the set of vertices $U_0\subset V(G)$ of color $\ell$ has size at least $n/r$. If $\ell=1$, it means that every element of $U_0$ is contained in $L$. But every $n/r$ vertices of $V(G)$ span $\mathbb{F}_q^k$, a contradiction. Hence, we may assume that $\ell\geq 2$. As $|U_0|\geq n/r\geq 4\lambda n/d$, we can apply Lemma \ref{lemma:large_component} to find $U\subset U_0$ of size at least $n/r-\frac{2\lambda n}{d}\geq n/(2r)$ such that $G[U]$ is connected. Let $|U|=m$, let $T$ be a spanning tree of $G[U]$, and let $x_1,\dots,x_m$ be an enumeration of the elements of $U$ such that $x_i$ is a leaf in the subtree $T[x_i,\dots,x_m]$. 

As $x_i$ is not $(\ell-1)$-full, there exists an $(\ell-1)$-tuple  $f_i\subset B_{\ell-2}(x_i)$ such that the span of $f_i$ intersects $L$ trivially. Let $F=\bigcup_{i=1}^m  f_i$. Note that the sets $f_1,\dots,f_{m}$ are not necessarily disjoint. However, $f_i$ is contained in $B_{\ell-2}(x_i)$, hence $f_i$ and $f_j$ are disjoint if the distance of $x_i$ and $x_j$ is more than $2(\ell-2)$. The number of vertices of distance at most $2\ell-4$ is at most $d^{2\ell-4}$, so $U$ contains at least $|U|/d^{2\ell-4}\geq |U|/d^{2s}$ vertices, any pair of which are at distance more than $2\ell-4$. This implies $|F|\geq |U|/d^{2s}\geq n/(4rd^{2s})$, so the elements of $F$ span $\mathbb{F}_q^k$. 

We show that   $H_{W\rightarrow L}[F]$ contains an $\ell$-tree-like spanning subhypergraph. For $i=1,\dots,m$, let $F_i=f_m\cup f_{m-1}\cup\dots\cup f_{m-i+1}$. Let $H_1$ be the empty hypergraph on vertex set $F_1=f_m$, then $H_1$ is $\ell$-tree-like. Having already constructed an $\ell$-tree-like hypergraph $H_i$ on vertex set $F_i$ for some $i\in [m-1]$, we construct $H_{i+1}$ as follows. As $x_{m-i}$ is a leaf in $T[x_{m-i},\dots,x_{m}]$, it has a unique neighbour $x_z$ in this tree with $z\in \{m-i+1,\dots,m\}$. For every $y\in f_{m-i}\setminus F_i$, consider the $\ell$-tuple $g=\{y\}\cup f_{z}$. Here, $y\in f_{m-i}\subset B_{\ell-2}(x_{m-i})$, but $x_{m-i}$ is a neighbour of $x_{z}$, so $y\in B_{\ell-1}(x_{z})$ as well. Also, $f_{z}\subset B_{\ell-2}(x_{z})$, so $g\subset B_{\ell-1}(x_{z})$. Thus, by (i), there is a non-zero linear combination of elements of $g$ contained in $L$. Let $g'_y\subset g$ be such that a proper linear combination of elements of $g'_y$ is in $L$, then $g'_y$ is an edge of $H_{W\rightarrow L}[F]$. Also, me must have $y\in g'_y$, otherwise $g'_y\subset f_{z}$, contradicting that the span $f_{z}$ intersects $L$ trivially. Define $H_{i+1}$ to be the hypergraph on $F_{i+1}=F_i\cup f_{m-i}$ we get by adding the edges $g'_y$ to $H_i$ for every $y\in f_{m-i}\setminus F_i$. Then $H_{i+1}$ is an $\ell$-tree-like subhypergraph spanning $H_{W\rightarrow L}[F_{i+1}]$. In conclusion, $H_m$ is an $\ell$-tree-like subhypergraph of $H_{W\rightarrow L}$ on vertex set $F_m=F$. But then by Lemma \ref{lemma:tree-like}, 
$$\dim(\mbox{pspan}(H_m)\cap L)\geq \dim(F)-\ell+1=k-\ell+1.$$
This is impossible if $\ell\leq s$, as the dimension of $L$ is $k-s$. Therefore, we must have $\ell=s+1$, in which case the previous inequality tells us that $B\cap L$ spans $L$. See Figure \ref{fig:3} for an illustration of the $\ell$-tree-like hypergraph $H_m$.
\end{proof}

\subsection{Optimal constructions for small primes}\label{sect:small_prime}

Finally, we present our construction for small prime powers $q$. In this setting, our goal is to show that there is an explicit strong $s$-blocking set of size $O_{s,q}(k)$, as we imagine $q$ being bounded by some function of $s$.

\medskip

\noindent
\textbf{Construction.} Let $r=s+1$, let $\varepsilon=\delta/2$, where $\delta$ is given by Lemma \ref{lemma:constr2}. Let $d=p+1$, where $p$ is the smallest prime congruent to 1 modulo 4 larger than $16q^{4s}/\varepsilon^2$. Let $\lambda=2\sqrt{d-1}$, then $\lambda/d<\varepsilon /(2q^{2s})$. 
\begin{itemize}
    \item Assume that $k$ is sufficiently large satisfying the requirements of Lemma \ref{lemma:constr2}, and sufficiently large with respect to $q$. Let $W_0\subset \mathbb{F}_q^k$ be a set of $n_0=O(k)$ vectors such that any $(1-\delta)n_0$ elements of $W_0$ spans $\mathbb{F}_q^k$.
    \item Let $n_0>n=(1+o(1))n_0$ for which there is an explicit $(n,d,\lambda)$-graph $G$. This exists by Lemma \ref{lemma:Ramanujan} (here, we are using the assumption that $k$ is sufficiently large with respect to $q$, so $n$ is sufficiently large with respect to $d$). Let $W$ be an arbitrary $n$ element subset of $W_0$, and associate the elements of $W$ to the vertices of $G$. Note that any $(1-\varepsilon)n$ elements of $W$ span $\mathbb{F}_q^k$.

    \item Let $H$ be the $r$-uniform hypergraph on vertex set $W$, where $f\in W^{(r)}$ is an edge if there is some $x\in W$ such that every element of $f$ is within distance at most one from $x$ in $G$.

    \item Let $B=\bigcup_{f\in H}\mbox{span}(f)$.
\end{itemize}

\begin{figure}
    \centering
    \begin{tikzpicture}[scale=0.8]

        \path[draw,very thick,pattern=horizontal lines] (-3.5,1) --(0,0.5) -- (0,2.5);
        \node at (-1.16,1.333) [fill=white]{$f_{z_1}$};
        \path[draw,very thick,pattern=vertical lines] (3.5,1) --(0,0.5) -- (0,2.5);
        \node at (1.16,1.333) [fill=white]{$f_{z_2}$};
      
        \draw (0,3) ellipse (2 and 1);
        \draw[rotate around={45:(4,1)}] (4,1) ellipse (1 and 2);
        \draw[rotate around={-45:(-4,1)}] (-4,1) ellipse (1 and 2);
        \draw (0,0) ellipse (2 and 1);

        \node[vertex,label=below:$x_{y_0}$] (y0) at (0,0.5) {};
        \node[vertex,label=above:$x_{y_1}$] (y1) at (0,2.5) {};
        \node[vertex,label=below:$x_{z_1}$] (z1) at (-3.5,1) {};
        \node[vertex,label=below:$x_{z_2}$] (z2) at (3.5,1) {};

        \draw[ultra thin] (z1) -- (y1) -- (z2) ;

        \node at (0,-1.4) {$V_{y_0}$};
        \node at (0,4.4) {$V_{y_1}$};
        \node at (-4.9,2) {$V_{z_1}$};
        \node at (4.9,2) {$V_{z_2}$};

         \node[vertex] (x1) at (-0.7,0.2) {};
         \node[vertex] (x2) at (-1.3,0.4) {};
         \node[vertex] (x3) at (-1,-0.4) {};
         \node[vertex] (x4) at (1.3,0.1) {};
         \node[vertex] (x5) at (1,-0.4) {};

         \draw[very thick] (y0) -- (x1) -- (x2);  \draw[very thick] (x1) -- (x3);
          \draw[very thick] (y0) -- (x4) -- (x5);

         \node[vertex] (x1) at (-0.7,2.8) {};
         \node[vertex] (x2) at (-1.3,3) {};
         \node[vertex] (x3) at (-1,3.3) {};
         \node[vertex] (x4) at (1.3,3.1) {};
         \node[vertex] (x5) at (1,2.8) {};
         \node[vertex] (x6) at (-0.5,3.6) {};

          \draw[very thick] (y1) -- (x1) -- (x2)  ;  \draw[very thick] (x1) -- (x3) -- (x6) ;
          \draw[very thick] (y1) -- (x5) -- (x4);

          \node[vertex] (x1) at (-4.2,0.8) {};
         \node[vertex] (x2) at (-4.5,0) {};
         \node[vertex] (x3) at (-3.5,1.8) {};

        \draw[very thick] (x3) -- (z1) -- (x1) -- (x2)  ;  

         \node[vertex] (x1) at (4.2,0.8) {};
         \node[vertex] (x2) at (4.5,0) {};
         \node[vertex] (x3) at (3.5,1.8) {};
        \node[vertex] (x4) at (4.5,1.3) {};

        \draw[very thick] (x3) -- (z2) -- (x1) -- (x2)  ;  
        \draw[very thick] (z2) -- (x4);

    \end{tikzpicture}
    \caption{An illustration for the proof of Theorem \ref{thm:small_q}.}
    \label{fig:2}
\end{figure}
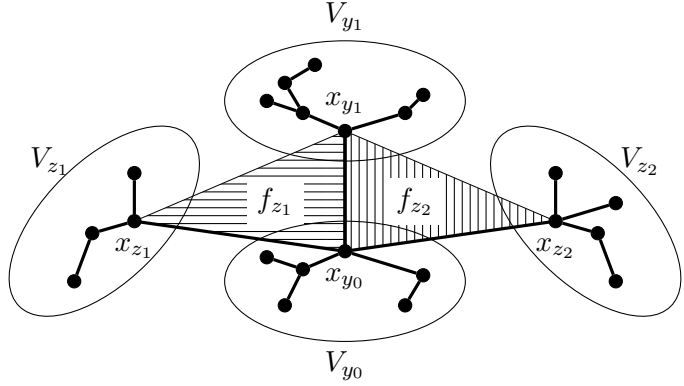

\begin{theorem}\label{thm:small_q}
$\tilde{B}$ is a strong $s$-blocking set of size $q^{O(s^2)} k$ in $P(k-1,q)$.
\end{theorem}

\begin{proof}
The size of $B$ is at most $q^{s+1}e(H)\leq q^{s+1}(2d)^r n \leq q^{O(s^2)}k$, so  $|\tilde{B}|=q^{O(s^2)}k$ as well.

We show that if $L<\mathbb{F}_q^k$ has codimension $s$, then $B\cap L$ spans $L$. Let $w_1,\dots,w_s$ be a basis of the orthogonal complement of $L$. For every $z\in \mathbb{F}_q^s$, let $U_z\subset V(G)$ be the set of vertices $x\in V(G)$ such that $\langle x,w_i\rangle =z(i)$. Then $U_z$ is the intersection of $W$ with a coset of $L$. In particular, the sets $U_z$ for $z\in  \mathbb{F}_q^s$ partition $V(G)$ into $q^s$ parts, and $U_0=L\cap W$. Observe that for every $z$, if $x,y\in U_z$, then $x-y\in L$. Hence, if $z\neq 0$, then $\{x,y\}$ is an edge of the proper linear combination hypergraph $H_{W\rightarrow L}$. In particular, every edge of $G[U_z]$ is an edge of $H_{W\rightarrow L}$. In general, we observe that if $x_1\in U_{z_1},\dots,x_{u}\in U_{z_u}$, then a linear combination $a_1 x_1+\dots+a_u x_u$ is contained in $L$ if and only if $a_1z_1+\dots+a_u z_u=0$.

Let $b=\lambda n/d$, and let $S\subset \mathbb{F}_q^s$ be the set of vectors $z$ such that $|U_z|>(q^s+2) b$. Furthermore, let $V_z\subset U_z$ be the vertex set of the largest component of $G[U_z]$, then $|V_z|\geq |U_z|-2b> q^s b$ by Lemma \ref{lemma:large_component}. Let $V'=\bigcup_{z\in S}V_z$, then
$$|V'|=\sum_{z\in S}|V_z|\geq n-\sum_{z\not\in S}|U_z|-\sum_{z\in S}2b\geq n-q^s (q^s+2)b-2bq^s>n(1-2q^{2s}\lambda/d)\geq n(1-\varepsilon).$$
Therefore, $V'$ spans $\mathbb{F}_q^k$. Next, our goal is to construct an $r$-tree-like subhypergraph of $H_{W\rightarrow L}$, whose every edge is contained in some edge of $H$.

If $S$ only contains the zero vector, then $V'=V_0$, and $V'\subset L$, which contradicts that $L$ does not span the whole space. Hence, we may assume that $S$ contains at least one non-zero vector. Let $Y\subset S$ be a basis of the subspace of $\mathbb{F}_q^s$ spanned by $S$. So $1\leq |Y|\leq s$, and select an arbitrary $y_0\in Y$.  Apply Lemma \ref{lemma:star} to the sets $(V_z)_{z\in S}$ with $V_{y_0}$ playing the role of $V_0$. We have $|V_{z}|\geq |S|\lambda n/d$ for every $z\in S$, so the condition of Lemma \ref{lemma:star} is satisfied. Hence, there exists a vertex $x_{y_0}\in V_{y_0}$ such that for every $z\in S\setminus \{y_0\}$, there exists some $x_z\in V_z$ such that $\{x_{y_0},x_z\}$ is an edge of $G$. Now for every $z\in S\setminus Y$, the $|Y|+1$ vectors $z,\{y\}_{y\in Y}$ are linearly dependent. More precisely, there is a linear combination summing to zero, where the coefficient of $z$ is not zero. Hence, $H_{W\rightarrow L}$ contains an edge $f_{z}$ such that $x_{z}\in f_{z}\subset \{x_{z}\}\cup \{x_y:y\in Y\}$. Observe that $f_{z}$ is also a subset of some edge of $H$. The union of these edges $f_{z}$ for $z\in S\setminus Y$ is an $(|Y|+1)$-tree-like subhypergraph of $H_{W\rightarrow L}$, whose union is the set of vertices $\{x_{z}:z\in S\}$. In order to cover the rest of the vertices of $V'$, we observe that $G[V_z]$ has a spanning tree $T_z$. But then there is an enumeration $v_1,\dots,v_{|V_z|}$ of the vertices of $T_z$ such that $v_i$ is a leaf in $T_z[v_i,\dots,v_{|V_z|}]$, and $v_{|V_z|}=x_z$. Now the edges $f_{z}$ for $z\in S\setminus Y$ together with the edges of the trees $T_z$ for $z\in S$ form an $(|Y|+1)$-tree-like subhypergraph $H_0$ of $H_{W\rightarrow L}$ on vertex set $V'$. This is witnessed by any ordering of $V'$ in which the vertices $\{x_y:y\in Y\}$ come last, preceded by any ordering of $\{x_z:z\in S\setminus Y\}$, and before that any ordering of the rest of the vertices which respects the orderings of the trees $T_z$. See Figure \ref{fig:2} for an illustration.

The proper span of $H_0$ is contained in $B$, and by Lemma \ref{lemma:tree-like}, 
$$\dim(\mbox{pspan}(H_0)\cap L)\geq \dim(V')-(|Y|+1)+1\geq k-s.$$
Therefore, we must have $|Y|=s$, and $B\cap L$ spans $L$, finishing the proof.
\end{proof}

\section*{Acknowledgments}
We would like to thank the organizers of DiGResS workshop at Ume\r{a}, 2024, where this research was initiated. Also, we would like to thank  Peter Beelen, Alessandro Neri, and Tovohery H. Randrianarisoa for valuable discussions, and the anonymous referee for many useful comments and suggestions.

\section*{Appendix --- Lemma \ref{lemma:constr1}}

We prove the missing part of Lemma \ref{lemma:constr1}. We provide a construction that is a concatenated code, whose outer code is a Reed-Solomon code, and inner code is a linear code meeting the Gilbert-Varshamov bound (found by brute-force search). See e.g. Section 12 of \cite{Roth} for the relevant definitions. Here, we give a self-contained proof without the use of coding terminology.

\begin{theorem}
    For every integer $s$ and $\varepsilon>0$, there exists $C$ such that the following holds. 
Let $q$ be a prime power and $k$ be an integer such that $k\geq \varepsilon q$, and both are sufficiently large with respect to $s$ and $\varepsilon$. Then there exists a set $W\subset \mathbb{F}_q^k$ of size at most $Ck$ such that any $s+1$ elements of $W$ are linearly independent, and any $\varepsilon |W|$ elements of $W$ span $\mathbb{F}_q^k$. Moreover, such a set $W$ can be constructed in $k^{O_{s,\varepsilon}(1)}$ time.
\end{theorem}

\begin{proof}
Choose  $R \in (0, 1)$ and an integer $M \ge 2$ strictly depending on $\varepsilon$, such that $R+ \frac{1}{M}< \frac{\varepsilon}{2}$.
Let $m$ be the minimal integer such that $m \ge s+2$ and $q^m \ge \frac{k}{mR}$. Define $Q = q^m$, and let $K = k/m$ (where we assume that $m$ divides $k$, the general case being almost identical, but slightly more technical), and  $N = \lceil K/R \rceil$. By our choice of $m$, we have $Q>N$, so we can choose a set $S = \{\alpha_1, \dots, \alpha_N\}$  of $N$ distinct elements of $\mathbb{F}_Q^*$.

Fix an arbitrary basis of $\mathbb{F}_Q$ over $\mathbb{F}_q$ and define the canonical linear isomorphism $\phi : \mathbb{F}_Q \rightarrow \mathbb{F}_q^m$.

\begin{claim}
There exists a linear mapping $L : \mathbb{F}_Q \to \mathbb{F}_{q}^{mM}$ of the form $L(x) = \left(\phi(\gamma_1 x), \dots, \phi(\gamma_M x)\right)$ such that for all $x \in \mathbb{F}_Q^*$, the  fraction of non-zero coordinates of $L(x)$ is strictly bounded below by $\delta_{0} = 1 - 2/M$. Moreover, if $v\in \mathbb{F}_q^{mM}\setminus\{0\}$ with at most  $s+1$ non-zero entries, then $\langle v,L(x)\rangle\neq 0$ for some $x\in \mathbb{F}_Q^*$.
\end{claim}
\begin{proof} Pick $\vec{\gamma} = (\gamma_1, \dots, \gamma_M) \in (\mathbb{F}_Q)^M$ randomly from the uniform distribution, and define $L_{\vec{\gamma}}$ accordingly. Let $d = \lfloor \delta_{0} mM \rfloor$. For any fixed non-zero $x \in \mathbb{F}_Q^*$,  $L_{\vec{\gamma}}(x)$ is uniformly distributed over $(\mathbb{F}_q^{m})^M$. The probability that $L_{\vec{\gamma}}(x)$ has less than $d$ non-zero entries is the volume of the Hamming ball $V_q(mM, d-1)$ divided by $q^{mM}$. We apply the union bound over all $Q-1$ non-zero elements of $\mathbb{F}_Q^*$
$$\mathbb{P} \left[ \exists x \neq 0 : |\supp(L_{\vec{\gamma}}(x))|<d\right] \le (q^m - 1) \frac{V_q(mM, d-1)}{q^{mM}} < q^m \frac{q^{mMH_q(\delta_{0})}}{q^{mM}},$$
where $$H_q(t) = t \log_q(q-1) - t \log_q(t) - (1-t) \log_q(1-t)$$ is the $q$-ary entropy function. As $\lim_{q\rightarrow\infty} H_q(t)=t$, in case $\delta_0=1-2/M$, we have $H_q(\delta_0)< 1-1.5/M$ for $q$ sufficiently large, so the right-hand-side is less than 1/2.

The number of vector $v$ with at most $s+1$ non-zero entries is at most $q^{s+1}(mM)^{s+1}$. The probability that for a fixed non-zero $v$, we have $\langle v,L_{\vec{\gamma}}(x)\rangle=0$ for every $x$ is $1/q^m$. Therefore, by the union bound
$$\mathbb{P} \left[ \exists v \neq 0 : |\supp(v)|\leq s+1\text{ and }\forall x: \langle v,L_{\vec{\gamma}}(x)\rangle=0\right] \le \frac{q^{s+1}(mM)^{s+1}}{q^m}.$$
Thus, if $m>s+1$ and $q$ is sufficiently large, the right-hand-side is less than $1/2$.
Hence, at least one valid  $\vec{\gamma}$ exists.
\end{proof}

As $Q$ is the smallest power of $q$ larger than $\max\{q^{s+2},k/(mR)\}$, we have $Q\leq k^{2s}$ for $k\geq \varepsilon q$ sufficiently large. Therefore, by searching all $M$-tuples $(\gamma_1,\dots,\gamma_M)$ of $\mathbb{F}_Q$, we can find the suitable linear mapping $L$ in $k^{O(sM)}$ time.

Let $\mathcal{P}_K$ be the $\mathbb{F}_Q$-vector space of polynomials with degree strictly less than $K$, the $\mathcal{P}_K\cong \mathbb{F}_q^{Km}=\mathbb{F}_q^k$. Define the $\mathbb{F}_q$-linear transformation $\Phi : \mathcal{P}_K \to \mathbb{F}_q^{NmM}$ as $$\Phi(P) = \left( L(P(\alpha_1)), L(P(\alpha_2)), \dots, L(P(\alpha_N)) \right).$$
Let $n = NmM$. By fixing an $\mathbb{F}_q$-basis for $\mathcal{P}_K$, the transformation $\Phi$ is represented by a matrix $A \in \mathbb{F}_q^{n \times k}$, such that $\Phi(y) = Ay$ for any coefficient vector $y \in \mathbb{F}_q^k$. We define $W \subset \mathbb{F}_q^k$ as the set of the $n$ columns of the transposed matrix $A^T$.

\begin{claim}
$|W| \le Ck$ for a constant $C$ depending only on $\varepsilon$.
\end{claim}

\begin{proof} We have $|W| = n = NmM$ and $k=mK$. Therefore, $|W|/k=NM/K\leq (K/R+1)M/K=M/R+M/K<M/R+1$. As $R$ and $M$ only depend on $\varepsilon$, this proves the claim.
\end{proof}

\begin{claim}
 If $k$ is sufficiently large with respect to $s,\varepsilon$, then any $s+1$ elements of $W$ are linearly independent over $\mathbb{F}_q$.
 \end{claim}
 
 \begin{proof} Assume to the contrary that this is false, then there exists a non-zero vector $v \in \mathbb{F}_q^n$ with at most $s+1$ non-zero entries such that $A^T v = 0$. This implies that for every $P\in \mathcal{P}_K$, we have $\langle v,\Phi(P)\rangle=0$. Recall that $\Phi(P)=(L(P(\alpha_1)),\dots,L(P(\alpha_N))$. Write $v=(v_1,\dots,v_N)$, then 
$\langle v,\Phi(P)\rangle=\sum_{i=1}^N \langle v_i,L(P(\alpha_i))\rangle$.
Let $j\in [N]$ such that $v_j\neq 0$. As $|\supp(v_j)|\leq s+1$, there exists some $x_j$ such that $\langle v_j,L(x_j)\rangle\neq 0$. Also, there are at most $s+1$ indices $i\in [N]$ such that $v_i\neq 0$. Therefore, we can find a polynomial of degree at most $s < K$ (which is satisfied if  $k$ is sufficiently large with respect to $s,\varepsilon$) such that $P(\alpha_j)=x_j$ and $P(\alpha_i)=0$ for all indices $i\neq j$ for which $v_i \neq 0$. But then $\langle v,\Phi(P)\rangle=\langle v_j,L(x_j)\rangle\neq 0$.
 \end{proof}

\begin{claim} 
Any subset of $\varepsilon|W|$ elements of $W$ spans $\mathbb{F}_q^k$
\end{claim} 
\begin{proof} 
If $U\subset W$ does not span $\mathbb{F}_q^k$, then there exists a non-zero $v\in \mathbb{F}_q^k$ orthogonal to all elements of $U$. Equivalently, $Av(u)=0$ for every $u\in U$. But observe that we can write  $Av=\Phi(P)$ for some $P\in \mathcal{P}_K$. Here, $\Phi(P)=(L(P(\alpha_1)),\dots,L(P(\alpha_N))$. As $P$ has degree less than $K$, we have $P(\alpha_i)=0$ for at most $K-1$ different indices $i\in [N]$. Moreover, if $P(\alpha_i)\neq 0$, then $L(P(\alpha_i))$ has at least $\delta_0$ non-zero coordinates. In total, we have
$$|U|\leq (K-1)mM+(1-\delta_0)n=\left(\frac{K-1}{N}+\frac{2}{M}\right)n\leq \left(2R+\frac{2}{M}\right)n\leq \varepsilon n.$$
\end{proof}

\end{proof}

\end{document}